\documentclass{amsart}
\usepackage{amsmath,amsthm,amssymb,amsfonts}
\usepackage[noadjust]{cite}
\usepackage{tikz}
\allowdisplaybreaks
\usepackage[colorlinks=true, allcolors=blue]{hyperref}

\usepackage[english]{babel}
\usepackage{url}
\usepackage{cite}

\providecommand{\U}[1]{\protect\rule{.1in}{.1in}}
\newtheorem{theorem}{Theorem}

\newtheorem{lemma}[theorem]{Lemma}

\newtheorem{proposition}[theorem]{Proposition}

\newcommand{\red}[1]{{\color{black} #1}}

\newcommand{\B}{\mathbb{B}}

\newcommand{\E}{\mathbb{E}}

\newcommand{\G}{\mathbb{G}}
\newcommand{\bS}{\mathbb{S}}

\newcommand{\M}{\mathbb{M}}
\newcommand{\N}{\mathbb{N}}

\renewcommand{\P}{\mathbb{P}}

\newcommand{\R}{\mathbb{R}}

\newcommand{\Z}{\mathbb{Z}}

\newcommand{\wt}{\widetilde}
\newcommand{\ovl}{\overline}
\newcommand{\ep}{\epsilon}

\newcommand{\cN}{\mathcal{N}}

\newcommand{\cX}{\mathcal{X}}

{}

{}

\numberwithin{equation}{section}
\numberwithin{theorem}{section}

\newtheorem{thm}{Theorem}[section]

\theoremstyle{definition}

\newtheorem{rem}[thm]{Remark}

\begin{document}

	\title[Synchronization in RGG on the sphere]{Phase Synchronization in Random Geometric Graphs on the 2D Sphere}
	
	\author[C. De Vita]{Cecilia De Vita}
	\address{Departamento de Matem\'atica\hfill\break \indent Facultad de Ciencias Exactas y Naturales\hfill\break \indent Universidad de Buenos Aires\hfill\break \indent IMAS-UBA-CONICET\hfill\break \indent Buenos Aires, Argentina}
	\email{cdevita@dm.uba.ar}
	
	\author[P. Groisman]{Pablo Groisman}
	\address{Departamento de Matem\'atica\hfill\break \indent Facultad de Ciencias Exactas y Naturales\hfill\break \indent Universidad de Buenos Aires\hfill\break \indent IMAS-UBA-CONICET\hfill\break \indent Buenos Aires, Argentina and\hfill\break \indent \hfill\break \indent  NYU-ECNU Institute of Mathematical Sciences\hfill\break \indent  at NYU Shanghai}
	\email{pgroisma@dm.uba.ar}
	
	\author[R. Huang]{Ruojun Huang}
	\address{Fachbereich Mathematik und Informatik\hfill\break \indent Universit\"at M\"unster, \hfill\break \indent Einsteinstr. 62, M\"unster 48149, Germany}
	\email{ruojun.huang@uni-muenster.de}
	
	
	 \date{\today}
	
	\begin{abstract}
    The Kuramoto model is a classical nonlinear ODE system designed to study synchronization phenomena. Each equation represents the phase of an oscillator and the coupling between them is determined by a graph. There is an increasing interest in understanding the relation between the graph topology and the spontaneous synchronization of the oscillators. Abdalla, Bandeira and Invernizzi \cite{abdalla2024guarantees} considered random geometric graphs on the $d$-dimensional sphere and proved that the system synchronizes with high probability as long as the mean number of neighbors and the dimension $d$ go to infinity. They posed the question about the behavior when $d$ is small. In this paper, we prove that synchronization holds for random geometric graphs on the two-dimensional sphere, with high probability as the number of nodes goes to infinity, as long as the initial conditions converge to a smooth function. \red{We conjecture a similar behavior for more general simply-connected closed Riemannian manifolds but we expect global synchronization to fail if the manifold is not simply-connected, as was shown in \cite{devita2025energy} and suggested in \cite{cirelli2024scaling}.}
    \end{abstract}

	\maketitle
	
	
	\section{Introduction}
	The Kuramoto model is a prototypical example to study synchronization phenomena that occur widely in science and technology \cite{mirollo1990synchronization, winfree1967biological, acebron2005kuramoto, bullo2020lectures, arenas2008synchronization, dorflerSurvey, strogatz2004sync, strogatz2000kuramoto}. Originally it was defined as a system of ordinary differential equations (ODE) with mean field coupling \cite{kuramoto1975self} but later on the relevance of understanding the system for different kinds of graphs became apparent \cite{Abdalla2022, abdalla2024guarantees, devita2025energy, GHV, girvan2002community,  kassabov2021sufficiently, kassabov2022global, taylor2012there}.
	
	Hence, the community considered the behavior of this system in circulant graphs \cite{wsg}, graphons \cite{Medvedev2014, MedvedevWgraphs, medvedev2018continuum}, small-world networks \cite{MedvedevSmallWorld}, strongly connected graphs \cite{kassabov2021sufficiently, kassabov2022global, taylor2012there}, Erd\H{o}s-R\'enyi graphs \cite{ling2019landscape, Abdalla2022, nagpal2024}, Random Geometric Graphs (RGG) in the Torus \cite{cirelli2024scaling} and in the $d$-dimensional sphere \cite{abdalla2024guarantees} among others.
	
    For a given (finite, possibly weighted) graph $\G=(V,\mathcal E)$ with adjacency matrix $A=(a_{ij})_{1\le i,j\le n}$, the Kuramoto model on $\mathbb G$ is the following system of ODEs
    	\begin{align*}
        \begin{cases}
       	\displaystyle{\frac{d}{dt}}u_i(t) = \omega_i + \sum_{j=1}^n a_{ij}\sin\left(u_j(t)-u_i(t)\right),  \red{\quad t>0,}\\[10pt]
			u_i(0)=u^0_i, \quad \quad i=1,2,...,n.
		\end{cases}
	\end{align*}
The parameters $\omega_i$ are called {\em natural frequencies}. We are interested in {\em phase synchronization} (stable states with $u_i=u_j$ for all $i,j$), which can occur only if $\omega_i=\omega$ for every $i$. We assume this throughout the manuscript. By changing variables to a moving frame $u_i \to u_i - \omega t$, we obtain
    	\begin{align}
		\begin{cases}\label{eq:kuramoto.general}
			\displaystyle{\frac{d}{dt}}u_i(t) = \sum_{j=1}^n a_{ij}\sin\left(u_j(t)-u_i(t)\right), \red{\quad t>0,} \\[10pt]
			u_i(0)=u^0_i, \quad \quad i=1,2,...,n.
		\end{cases}
	\end{align}
\red{Here, $a_{ij}=a_{ji}\ge 0$ are symmetric non-negative edge weights that measure the strength of connection between every pair of vertices $i,j$, and $a_{ij}>0$ if and only if $i,j$ are neighbors. The initial condition $u^0_i$ takes values in the unit circle $\mathbb S^1=\R/2\pi\Z$, as are $u_i(t)$ for $t>0$.}
    Since the system is invariant under shifts, it is convenient to assume that the initial condition (and hence the solution for every time) verifies $\sum_{i=1}^n u_i^0=0$. This means that the dynamics takes place on the subspace orthogonal to $(1,1, \dots, 1)$. 

    Equation \eqref{eq:kuramoto.general} admits a potential given by
    \begin{equation}
    \label{eq:energy}
    \mathsf E(u_1, \dots, u_n) = \frac12 \sum_{i,j=1}^n a_{ij}\big[1 - \cos(u_j-u_i)\big].
    \end{equation}
    It is straightforward to check that $\dot u_i = -\partial \mathsf E/\partial u_i$. For a system like \eqref{eq:kuramoto.general} we say that there is {\em spontaneous synchronization} if the only stable equilibrium of the system is the phase-locked state $u_i=0$ for every $1\le i \le n$. This is equivalent to the energy function $\mathsf E$ having a unique global minimum. In this case, with the exception of a zero-measure set, all the initial conditions converge towards the global minimum.

 Here and in the literature {\em spontaneous synchronization}, {\em global synchronization} and {\em benign landscape} (for $\mathsf E$) are used indistinguishably. The last name comes from the nonlinear optimization community, for which the interest in this problem arises naturally from the possibility of understanding the geometry of a nonconvex functional to be optimized.

 \red{A random geometric graph (RGG) is a graph with vertex set given by a set of random points in some prescribed metric space and edge set formed by some geometric considerations, usually closeness with respect to the metric. Such graphs have a spatial structure, in stark contrast with Erd\H os-R\'enyi random graphs which are formed in a mean-field way. In particular, the spatial structure has been shown to be responsible of some particular behaviors in systems of coupled oscillators. That is the case, for example, of the Kuramoto on cycle graphs that support {\em twisted states} \cite{wsg} and the RGG in the circle. RGGs in the circle have the same topology as the cycle graph but, opposite to them, they are not exactly solvable due to the lack of symmetries that are present in the cycle \cite{devita2025energy}.
 
 In this article, for every $n\in\N$ and $\ep>0$, we consider RGG with vertex set given by $n$ points drawn i.i.d. uniformly on the $2$-sphere $\mathbb S^2$, and edge set formed by connecting every point with any other point within $3$ dimensional Euclidean distance $\sqrt{\ep}$ (here we view $\mathbb S^2\subset\R^3$). We will also put weights on the edges, rendering it a weighted RGG, but we do not go into it now.}
  
Abdalla, Bandeira, and Invernizzi \cite{abdalla2024guarantees} proposed studying the Kuramoto model in RGG on the $d$-sphere. 
They show that spontaneous synchronization occurs with high probability as $n\to \infty$ in at least one of the two regimes (see \cite[Theorems 3 and 4]{abdalla2024guarantees}):
    \begin{align*}
     np\ge  C_1(\log n)^{10}, \quad d \ge C_2(\log n)^3,
    \end{align*}
    or 
     \begin{align*}
     np\ge  C_1(\log n)^{2}, \quad d \ge C_2(n^2p^2+(\log n)^4)(\log n)^4, \quad \frac{c_0}{n}< p< \frac{1}{2},
    \end{align*}
    for any $c_0>0$ and some finite constants $C_i=C_i(c_0)$, $i=1,2$. Here $p$ represents the probability that two independent uniform points in $\bS^d$ are within Euclidean distance $\sqrt{\ep}$ (i.e.~are neighbors). This probability depends on both $\ep$ and $d$, and can be written as $p=\frac{1}{2}I_{\ep(1-\frac{\ep}{4})}(\frac{d}{2},\frac{1}{2})$, where $I_x(a,b)$ is the regularized incomplete beta function (see \cite[pp. 2]{lee2014concise}).

    The techniques developed in the paper \cite{abdalla2024guarantees} -- according to the authors -- do not apply to deal with the case in which $d$ is small and in particular to the case $d=2$, which is of special interest due to its role to model physical space.

    \red{The main feature of their condition is the diverging of dimension $d$ of the sphere with $n$. As explained in \cite{liu2021} (on which \cite{abdalla2024guarantees} heavily relies), the intuition behind this case is that when dimension is sufficiently high with respect to $n$, the random geometric graphs become increasingly indistinguishable from an Erd\H os-R\'enyi random graph. Considering the extreme case that $n$ is fixed and $d\to\infty$ and identifying the $n$ points on $\bS^{d-1}$ with vectors $v_1,..., v_n\in\R^d$, then these $n$ vectors are essentially mutually orthogonal. That is, the inner products $\langle v_i,v_j\rangle$ are essentially mostly close to $0$. Since the criterion for $(i,j)$ to be neighbors in the RGG can be rephrased as $\langle v_i,v_j\rangle$ exceeding a threshold very close to $1$, the fact that $(i,j)$ are neighbors has very little impact on whether $(i,k)$ or $(j,k)$ are neighbors, because they will typically be very close to $0$. Given this, it is not hard to believe that in some regime where $d,n$ both grow with $d=d(n)$, this intuition could still hold true and we can still compare with Erd\H os-R\'enyi graph, and transport what we know about Kuramoto model on these graphs to say things about RGG on high dimensional spheres. }

    In this work, we deal with the case $d=2$ in the regime	
		\begin{align}
        \label{ep-regime}
					\ep=\ep(n)\to 0 \quad \text{as }n\to\infty,\qquad  
\liminf_{n\to\infty}\frac{\ep^2n}{\log n}=\infty.
		\end{align}
    \red{In this case, $p=\pi\ep$ which follows from a simple computation of the area of a spherical cap.
    
    Our regime and the regime of \cite{abdalla2024guarantees} are not comparable, since we are in fixed low dimension, and they are in a large $d$ limit case. Neither result implies the other. Moreover, the reasons that lead to global synchronization seem to be different for $d=2$ than for $d\to \infty$. As mentioned before, as $d\to \infty$ the graph looks similar to an Erd\H os-R\'enyi graph, for which global sync is known. For $d=2$ our graph $\mathbb G_n$ is substantially different. The spatial structure  develops dependencies between points close to each other and the geometry of the sphere (being simply-connected) plays a key role.}

    We prove that synchronization occurs with high probability as $n\to \infty$ if the initial conditions converge to a smooth initial function defined on $\mathbb S^2$. A precise statement is given in the next section. We remark that this does not imply a global synchronization result since we are saying nothing about sequences of initial conditions that do not converge to a smooth function. 

    Our proof is based on a scaling limit that states that, in our regime, solutions to \eqref{eq:kuramoto.general} converge, in compact time intervals, to solutions of the heat equation on the sphere with values in $\mathbb S^1$, if the initial conditions do so. Since this heat equation is globally synchronizing, we can use our scaling limit to ensure that the solution of \eqref{eq:kuramoto.general} visits a neighborhood of the phase-locked state in finite time. The argument concludes with the use of a well-known result \cite{bullo2020lectures} that guarantees that for any connected graph of $n$ vertices, the set $\{\mathbf u\in (\bS^1)^n \colon |u_i -u_j| < \pi/2\}$ is contained in the basin of attraction of the phase-locked state. This strategy has been previously used at least in \cite{GHV,nagpal2024}.

    \red{We add a word on the role of $\ep$ in our strategy of proof. As said above, $\ep$ controls the size of the interaction of each Kuramoto oscillator, equivalently the size of the neighborhood of each random point on the sphere. Since we are interested in the case of localized interactions or sparse graphs, we are thinking of $\ep$ so small that the order of the average vertex degree $\ep n$ does not grow with $n$ too fast (in general, the faster $\ep n$ grows, the easier is analysis, as we approach the mean field). The most extreme case is where $\ep n$ remains order $1$, but that is far from what we can achieve; nevertheless, it is instructive to think about that case: a continuum analogue of nearest-neighbor interaction or something a bit more relaxed than that. In order to connect to a local PDE like the heat equation, it is necessary that $\ep\to 0$ as $n\to\infty$. In this paper, we can cover a regime of $\ep$ so that $\ep\to0$ but not too fast, which is captured in \eqref{ep-regime} (in particular, the average vertex degree diverges faster than $\ep^{-1}\log n$). In most of our proofs, the two parameters $\ep,n$ are treated as two independent parameters, and our bounds are given quantitatively in terms of them. From such quantitative bounds one can see what relation they need to satisfy in order for certain probabilities of bad events to be sufficiently small. }

    Several recent works, starting from \cite{ling2023local} and then extended in \cite{endor2024benign, mcrae2025benign}, give sufficient conditions that guarantee the energy landscape of a given Kuramoto model is benign, by checking an inequality on the condition number of the associated graph Laplacian (\cite[Theorem 2.2]{endor2024benign}) or certain normalized Laplacian matrix (\cite[Theorem 2.1]{mcrae2025benign}). \red{Here, the condition number of a (nonnegative definite) Laplacian matrix is the ratio between its largest eigenvalue and smallest nonzero eigenvalue (cf. \cite[Theorem 2.2]{endor2024benign}).}
    The results are more general, cast in the Burer-Monteiro factorizations of MaxCut-type semidefinite programs, which include Kuramoto model as a special case. However, we note that this sufficient condition, which requires the condition number of the said matrices to be strictly less than $2$, is unlikely to apply to our RGG case; indeed we expect the condition number in our sparse RGG graphs in small space dimensions to diverge to infinity with $n$. See also \cite[Section 3.4.5]{mcrae2025benign} for a very relevant discussion.

   \smallskip
	The paper is organized as follows. In Section \ref{sec:setup} we give precise definitions and state our main results. Section \ref{sec:integral.equation} deals with an integral equation that we use to approximate both \eqref{eq:kuramoto} and the heat equation \eqref{eq:heat} to obtain the scaling limit. This equation has previously been considered in \cite{cirelli2024scaling} with the sphere $\mathbb S^2$ replaced by the $d$-dimensional torus. Existence and uniqueness of the solutions, as well as the regularity theory can be handled in a similar way here, with the adequate caveats to deal with the curvature of the sphere. Although similar, we include the proofs for the reader's convenience. In Section \ref{sec:integral.equation} we prove that these approximations in fact do the job. Finally, in Section \ref{sec:sync} we prove our synchronization result. 

\section{Main results and sketch of the proofs}
\label{sec:setup}
    
	Let $\bS^2:=\{x\in\R^3:\, \|x\|=1\}$ denote the unit sphere embedded in $\R^{3}$ endowed with its surface area measure $\sigma(\cdot)$, where $\|\cdot\|$ denotes the Euclidean distance in $\R^3$. We denote by $\B_r(x)$ the Euclidean ball of radius $r$ in $\R^3$ centered at $x$.
	
	For every $n\in\N$, consider $n$ points $V:=\{x_1,x_2,...,x_n\}$ on $\bS^2$, independently and identically distributed (i.i.d.) according to the uniform distribution with respect to $\sigma(\cdot)$. \red{We assume all the random quantities defined through the paper are defined on a given probability space $(\Omega, \mathcal F, \P)$ and we use $\E$ to denote expectation respect to $\P$. All the almost-sure statements are respect to $\P$.}

	We are also given a parameter $\ep=\ep(n)>0$ that depends on $n$.
    In addition, let $K: \R_+\to\R_+$ be a bounded function with compact support in $[0,1]$ such that $K(r)>0$ for every $r\in[0,1)$. We assume either that $K$ has bounded derivatives or that $K$ is the indicator function of $[0,1]$. We define the (weighted) random geometric graph $\mathbb G_n = (V,\mathcal E)$ with vertex set $V$ and edge set $\mathcal E$, by imposing $e=\{x_i, x_j\}\in \mathcal E$ if and only if $\|x_i-x_j\|^2<\ep$, in which case the edge weight is $w_e:=K\Big(\frac{\|x_i-x_j\|^2}{\ep}\Big)>0$. Let,
	\begin{align*}
		\cN_i:=\left\{j: j\neq i, \|x_i-x_j\|^2 <\ep\right\}\subset \{1,2,...,n\}\backslash\{i\}.
	\end{align*}
    be the set of neighbors of point $i$ in $\G_n$ and consider the random variable $N_i:=\text{Card}(\cN_i)$ representing the number of neighbors of that node (Figure \ref{fig:RGG}). We call $\mathbb G_n$ a Random Geometric Graph (RGG) on $\mathbb S^2$ with parameters $(K,\ep)$.

	The (homogeneous) Kuramoto model formed on $\G_n$ is a finite system of $n$ ordinary differential equations (ODE), with $u^n:[0,\infty)\times V\to\bS^1:=\R/2\pi$ such that,
	\begin{align}
		\begin{cases}\label{eq:kuramoto}
			\displaystyle{\frac{d}{dt}}u^n(t,x_i)&= \displaystyle{\frac{1}{\ep\E(N_i)}}\sum_{j=1, j\neq i}^n\sin\left(u^n(t,x_j)-u^n(t,x_i)\right)K\Big(\frac{\|x_i-x_j\|^2}{\ep}\Big), \\[10pt]
			u^n(0,x_i)&=u_0^n(x_i), \quad \quad i=1,2,...,n.
		\end{cases}
	\end{align}
    
     \begin{figure}
    \begin{center}
    \includegraphics[width=.75\textwidth]{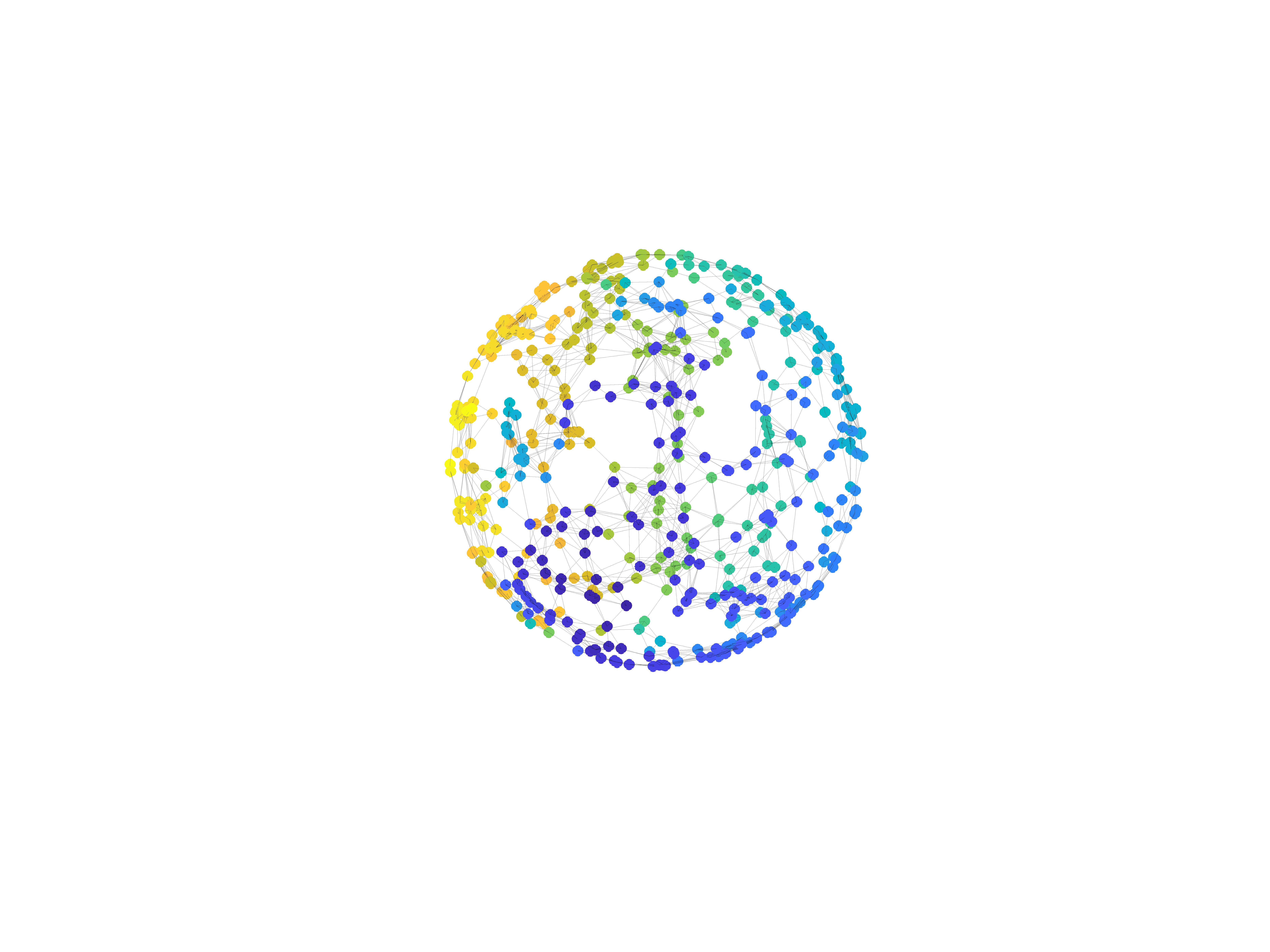}
    \end{center}
    \caption{A connected random geometric graph on $\bS^2$, where each node supports a Kuramoto oscillator whose phase on $\bS^1$ is colored.}
    \label{fig:RGG}
    \end{figure}
    	For every $n$ and realization of the random points $V$, there is a unique solution to \eqref{eq:kuramoto} since it is a finite system of ODEs with Lipschitz coefficients.
	
		In \eqref{eq:kuramoto}, we renormalize the sum on the right-hand side by $\E(N_i)$ instead of $N_i$ since the former is a deterministic constant independent of $i$, and the resulting ODE has a gradient flow structure.
		
		\red{We also note that in the literature on {\em homogeneous} Kuramoto models, there seems to be no ``standard" convention what the preconstant in front of the sum in \eqref{eq:kuramoto} should be. This is natural because one can change this preconstant by a time-change of the system, which clearly does not affect synchronization (or not). The specific form of the constant we choose (namely $\frac{1}{\ep\E(N_i)}$) is ultimately dictated by our desire to perform scaling limit to the heat equation; in particular, the $\frac{1}{\ep}$ comes from the usual second-order difference approximation when ``grid spacing" is $\sqrt{\ep}$. We use this technique to say something for the small $\ep$ / large $n$ models, whereas other papers that do not use this technique may adopt a  different preconstant; as said above, for fixed $\ep,n$ they are time-changes of each other.

        For definiteness we compare our preconstant with those adopted in e.g. \cite{abdalla2024guarantees, cirelli2024scaling, devita2025energy, kassabov2022global}. Our normalization constant $\frac{1}{\ep\E(N_i)}$ is actually in agreement with \cite{cirelli2024scaling}, with the caveat that the random number $N_i$ was used there instead of $\E(N_i)$, and points within distance $\ep$ are connected instead of our $\sqrt{\ep}$. That is why we see the seemingly different normalizing constant $\frac{1}{\ep^2 N_i}$ in \cite{cirelli2024scaling}. In \cite{devita2025energy}, there is an extra $\frac{1}{n}$. In \cite{abdalla2024guarantees}, the constant is just $\frac{1}{n}$, while in \cite{kassabov2022global} the constant is $1$. As said above, for homogeneous models, this constant can be adjusted by a time-change and does not affect the set of stable equilibria of the system.

        Up to our knowledge, the article \cite{cirelli2024scaling}, to which the present paper is closely related, seems to be the first one where one connects the Kuramoto model to the heat equation {\it{and at the same time}} exploits this connection to say things about the synchronization of the Kuramoto dynamics. The scaling limit {\it per se} belongs to a larger literature of continuum limits of discrete dynamics on graphs that includes the Kuramoto model and many more. The precise details of the interaction can vary greatly; some are mean-field interactions, others more localized like ours, and yet others have extra randomness coming from Brownian motions or consider graphons. The series of work of Medvedev and collaborators are prominent examples, see e.g. \cite{MedvedevWgraphs, medvedev2018continuum} and references therein. The approach of using scaling limits to deduce synchronization for large systems of oscillators also appears in the literature. It goes back at least to \cite{nagpal2024}, in which a graphon approach is used to obtain the scaling limit of Kuramoto systems as well as other models of coupled oscillators. The authors use the scaling limit to prove synchronization in Erd\H os-R\'enyi graphs. We remark that in the work \cite{nagpal2024} there is no spatial structure and the edges are independent, as in all the works that deal with graphon models. That is not the case in our RGG. In our case, the spatial structure is responsible for the presence of the Laplacian / heat equation on the sphere and can not be put in the $W-$graphon setting with a smooth $W$. This seems to appear for the first time in \cite{cirelli2024scaling}.
    
        }


		For $x,y\in\bS^2$, their geodesic distance is defined as $\rho(x,y)=\arccos \langle x, y\rangle,$ where $\langle\cdot, \cdot\rangle$ denotes the inner product in $\R^3$. We have 
		\[
		\|x-y\|=\sqrt{2-2\langle x, y\rangle}=\sqrt{2-2\cos \rho(x,y)}.
		\]
		If $x,y$ are very close, then by Taylor expansion of cosine function near $0$, we have $\|x-y\|\approx \rho(x,y)$. Hence, it does not make much difference whether we use the geodesic distance or the Euclidean distance to construct the random geometric graph. In many contexts (e.g.~machine learning) the geodesic is not known in advance, so it is preferable to consider Euclidean distance. 
	
	Since the points $x_1,...,x_n$ are i.i.d. uniform, we have
	\begin{align}\label{exp-degree}
		\E(N_i)=\frac{\sigma\left(\B_{\sqrt{\ep}}(x_i)\cap\bS^2\right)}{\sigma(\bS^2)}n=\frac{\mathsf c_2 \ep n}{4\pi},
	\end{align}
	for some explicit constant $\mathsf c_2>0$ which is independent of $i$. In our regime \eqref{ep-regime}, $\E(N_i)$ goes to $\infty$ as $n\to\infty$ and in fact a bit more holds: the normalizing factor $\ep\E(N_i)$ on the right-hand side of \eqref{eq:kuramoto} diverges faster than $\log n$.

    The condition \eqref{ep-regime} (up to the logarithmic factor) coincides with the threshold for pointwise convergence of graph Laplacian to the Laplace-Beltrami operator on Riemannian manifolds, as it appears in machine learning literature, cf. \cite[Eq. (1.7)]{singer2006graph} (taking $d=2$ there). The convergence of the graph Laplacian in RGG on manifolds towards its manifold counterpart, the convergence of solutions of the discrete Laplace equation towards the continuous one and the convergence of spectral properties (eigenvectors, eigenvalues, spectral clustering) have been extensively studied \cite{coifman2006diffusion, trillos2016continuum, trillos2020error, calder2019game, Berry, Belkin, Hein} due to their prominent relevance in several areas, including machine learning, partial differential equations, differential geometry, calculus of variations, probability, and more.

        Here, besides the convergence of the Laplacian, we need to deal not just with the time evolution of the solution but also with the nonlinearity given by the sine function. The intuition is that the argument of the sine is typically very small (on the order of $\sqrt{\ep}$ if the solution is smooth), hence by Taylor expansion of sine at $0$, \red{we can think of our nonlinear operator as very close to a graph Laplacian, in the limit $\ep\to0$, and they both can be approximated by the continuous Laplacian (see in particular Proposition \ref{conv:interm-to-heat} for how we make the approximation)}. Rigorously proving this, and in a parabolic framework, constitutes the main bulk of our work.
	
	Recall Bernstein's concentration inequality: {\it Let $Y_1,...,Y_n$ be independent mean-zero random variables such that $|Y_j|\le 1$ a.s. Let $S_n=\sum_{j=1}^n Y_j$ and $\lambda>0$. Then we have that 
			\begin{align*}
				\P\left(|S_n|>\lambda\right)\le 2e^{-\frac{\lambda^2/2}{\sum_{j=1}^n\E(Y_j^2)+\lambda/3}}.
	\end{align*}}
        Applying it to $\mathsf Y_j:=1_{\{j\in\cN(i)\}}-\E[1_{\{j\in\cN(i)\}}]$, $j\in\{1,2,...,n\}\backslash\{i\}$  we have that
	\begin{align}\label{bernstein}
		\P\left(|N_i-\E(N_i)|>\lambda\right)\le 2e^{-\frac{\lambda^2/2}{\E(N_i)+\lambda/3}},
	\end{align}
	where we used that $\E(\mathsf Y_j^2)\le \E[1_{\{j\in\cN(i)\}}]$. Taking $\lambda=\delta\E(N_i)$, $\delta>0$, we have that 
	\begin{align}\label{concen}
		\P\big((1-\delta)\E(N_i)\le N_i\le (1+\delta)\E(N_i)\big)\ge 1- 2e^{-\frac{3\delta^2}{2\delta+6}\E(N_i)}.
	\end{align}
	Note that $\E(N_i)=\mathsf c_2 \ep n\gg \ep^2 n\gg \log n$ by \eqref{ep-regime}, hence $N_i$ is highly concentrated around its mean.
	
We will prove that the scaling limit of \eqref{eq:kuramoto} is given by the heat equation on the sphere, $u:[0,\infty)\times\bS^2\to\bS^1$:
	\begin{align}\label{eq:heat}
		\begin{cases}
			\displaystyle{\frac{d}{dt}}u(t,x) &=\kappa\, \Delta_{\bS^2} u(t,x) \\[5pt]
			u(0,x)&=u_0(x),
		\end{cases}
	\end{align}
	where $\Delta_{\bS^2}$ denotes the Laplace-Beltrami operator on $\bS^2$, and
	\begin{align}\label{cst:kappa}
		\kappa:=\frac{1}{4}\int_{\R^2}\|z\|^2K(\|z\|^2)dz,
	\end{align}
	where (by an abuse of notation) $\|\cdot\|$ is the Euclidean distance in $\R^2$. By Proposition \ref{ppn:equiv} below, a continuous function from $\bS^2$ to $\bS^1$ can be thought of as a continuous function from $\bS^2$ to $\R$. Observe that this is not the case when the manifold is not simply connected \cite{cirelli2024scaling}. Hence, our main result is a scaling limit from \eqref{eq:kuramoto} to \eqref{eq:heat} for $\R$-valued functions. We insist on viewing all our equations as equivalently taking values in $\R$ since our proof relies on comparison principles, namely, on $\R$ there is a natural ordering which is not the case on $\bS^1$. We denote with $C^{k,\alpha}(\mathbb M)$ the space of functions from $\mathbb M$ to $\R$ with continuous derivatives up to order $k$, all of them being $\alpha$-H\"older continuous. Similarly, we use $C^k(\mathbb M)$ when we do not require the H\"older continuity and $C(\mathbb M)$ for the space of continuous functions on $\M$. We omit writing $\mathbb M$ when it is not nesessary. For $u_0\in C^{2,\alpha}(\bS^2)$ for some $\alpha\in(0,1)$, there exists a unique $C^{1+\alpha/2, 2+\alpha}([0,\infty)\times\bS^2)$ solution to \eqref{eq:heat}, cf. \cite{krylov}.

    We are ready to state the scaling limit of \eqref{eq:kuramoto} towards \eqref{eq:heat}.
	
\begin{theorem}\label{thm.main}\red{
		Let $\mathbb G_n=(V,\mathcal E)$ be a RGG on $\mathbb S^2$ with parameters $(K,\ep)$. Fix $T>0$ and consider {$u^n: [0,T]\times V\to\R$} the unique solution of \eqref{eq:kuramoto} with initial condition $u_0^n:V\to\R$. Let $u:[0,T]\times\bS^2\to\R$ be the unique solution of \eqref{eq:heat} with initial condition $u_0\in C^{2,\alpha}(\bS^2)$ for some $\alpha\in(0,1)$. Assume \eqref{ep-regime} holds and $K$ is defined as in the beginning of the section. If}
		\begin{align}\label{cond:ini}
			\sum_{n=1}^ \infty  \P\Big(\sup_{1\le i\le n}|u^n_0(x_i) - u_0(x_i)|> \delta \Big)<\infty
		\end{align}
		for every $\delta>0$, then,
		\begin{equation*}
			\lim_{n\to \infty }\sup_{1\le i\le n}\sup_{t\in[0,T]} |u^n(t,x_i) - u(t,x_i)| = 0, \quad \text{almost surely.}
		\end{equation*}
	\end{theorem}

\begin{rem}
\red{Condition \eqref{cond:ini} guarantees the almost sure convergence of $u_0^n$ towards $u_0$ by means of Borel–Cantelli Lemma. Without this convergence we can not expect the conclusion of the theorem to hold. If we replace \eqref{cond:ini} with $\P\Big(\sup_{1\le i\le n}|u^n_0(x_i) - u_0(x_i)|> \delta \Big) \to 0$, we can obtain a similar statement but with convergence in probability instead of almost surely.

Even if $u_0^n$ is deterministic, the quantities $u_0^n(x_i)$ are random since the points $x_i$ are, but also other sources of randomness are allowed for $u_0^n(x_i)$ as far as \eqref{cond:ini} holds. }
\end{rem}

\begin{proof}
    The theorem follows by combining Proposition \ref{conv:interm-to-heat} and Proposition \ref{ppn:kura-int} below.
\end{proof}
Proposition \ref{conv:interm-to-heat} deals with an intermediate equation that interpolates between the ODE \eqref{eq:kuramoto} and the heat equation \eqref{eq:heat}. The equation is given by $u^{I,\ep}:[0,\infty)\times\bS^2\to\R$:
	\begin{equation}
		\begin{cases}\label{eq:intermediate}
			\displaystyle{\frac{d}{dt}}u^{I,\ep}(t,x) &=\displaystyle{ \frac{1}{\mathsf c_2 \ep^2}}\int_{\bS^2}\sin\left(u^{I,\ep}(t,y)-u^{I,\ep}(t,x)\right)K(\ep^{-1}\|x-y\|^2)d\sigma(y)\\[7pt]
			u^{I,\ep}(0,x)&=u_0(x).
		\end{cases}
	\end{equation}

\begin{proposition}\label{conv:interm-to-heat}
   Let $u:[0,\infty)\times\bS^2\to\R$ be the solution of \eqref{eq:heat} and $u^{I,\ep}:[0,\infty)\times \bS^2\to\R$ the solution of \eqref{eq:intermediate} with $u_0\in C^{2,\alpha}(\bS^2)$ for some $\alpha\in(0,1)$. Then for any $T>0$ there exist $C=C(T, {}{\alpha, u_0})>0$ and $\ep_0\in(0,1)$ such that for any $\ep<\ep_0$,
		\begin{align*}
			\left\|u^{I,\ep}-u\right\|_{L^\infty([0,T]\times\bS^2)}\leq C\ep^{\alpha/2}.
		\end{align*} 
\end{proposition}

The following proposition is proved in \cite{cirelli2024scaling} with $\mathbb S^2$ replaced by the $d$-dimensional torus. The proof for this case is essentially the same and we do not include it.
    
	\begin{proposition}[{\cite[Proposition 3.1]{cirelli2024scaling}}]\label{ppn:kura-int}
		Assume \red{as in \eqref{ep-regime} that $\epsilon = \epsilon(n) \to 0$ as $n \to \infty$  and $\liminf\limits_{n \to \infty} \frac{\epsilon^2 n}{\log n} = \infty$}.  If $u^{n}$ is the unique solution of \eqref{eq:kuramoto}, $u^{I,\ep}$ is the unique $C([0,\infty), C^1(\bS^2))$ solution of \eqref{eq:intermediate},  $u_0\in C^1(\bS^2)$ and  \eqref{cond:ini} holds, we have that
		\[
		\lim_{n\to\infty}\left\|u^{n}-u^{I,\ep}\right\|_{L^\infty([0,T]\times V)}=0, \quad a.s.
		\]
	\end{proposition}

With Theorem \ref{thm.main} at hand, we turn to the study of synchronization on the sphere.

	\begin{theorem}
		\label{thm:sync} \red{Under the conditions of Theorem \ref{thm.main},} let $A_n$ be the event that equation \eqref{eq:kuramoto} with initial condition $u_0^n$ achieves phase synchronization. Then
		\[
		\red{\P(A_n \, \, \text{\rm occurs for all $n$ large enough}) = 1.}
		\]
	\end{theorem}

The proof of this theorem is given in Section \ref{sec:sync}. 

\section{The integral equation}
\label{sec:integral.equation}

In this section, we first prove that we can actually work with functions with values in $\R$ rather than in $\mathbb S^1$ and next, we prove Proposition \ref{conv:interm-to-heat}. 

\subsection{Equivalence between $\bS^1$-valued and $\R$-valued functions}
	A continuous function $f \colon \mathbb{S}^2 \to \mathbb{S}^1$ can be interpreted as a function with values in $\R$ by identifying each point of $\mathbb{S}^1$ with its argument in $[0,2\pi)$. However, at first, this new function might not seem continuous, since approaching the point $(1,0)$ from different directions could yield arguments $0$ or $2\pi$. Nevertheless, we show that this issue does not arise because $\mathbb{S}^2$ is simply connected. This is a well-known topological fact.  
	
	\begin{proposition}
		\label{ppn:equiv}
		For any continuous function $f\colon \mathbb S^2 \to \mathbb S^1$ there exists a continuous function $\bar f \colon \mathbb S^2 \to \R$ such that $f(x)= (\cos(2\pi \bar f(x)), \sin(2\pi \bar f(x)))$.
	\end{proposition}
    
	\begin{proof}
		Let $p(t) = (\cos(2\pi t), \sin(2\pi t))$ be the universal covering of the unit circle.
		Consider $x_0$ to be a fixed, arbitrary point in $\mathbb{S}^2$, and let $y_0 = f(x_0) \in \mathbb{S}^1$. Choose $\bar{y_0} \in \R$ such that $p(\bar{y_0}) = y_0$.
		
		Let $x \in \mathbb{S}^2$ and $\gamma$ be a path in $\mathbb{S}^2$ from $x_0$ to $x$. Then, $f \circ \gamma$ is a path in $\mathbb{S}^1$ from $y_0$ to $f(x)$. Since $p$ is the universal covering map, there exists a unique lift $\overline{f \circ \gamma} : [0,1] \to \R$ of the path $f \circ \gamma$ (i.e.~$p\circ \overline{f \circ \gamma} = {f \circ \gamma}$) that starts at $\bar{y_0}$. In this way, we define $\bar{f}(x) := \overline{f \circ \gamma}(1)$. To show this is well defined, independently of the choice of $\gamma$, let $\gamma'$ be another path in $\mathbb{S}^2$ from $x_0$ to $x$. Since $\mathbb{S}^2$ is simply connected, $\gamma$ and $\gamma'$ are homotopic paths with fixed endpoints. That is, there exists a homotopy $H : [0,1] \times [0,1] \to \mathbb{S}^2$ such that $H(s,0) = \gamma(s)$, $H(s,1) = \gamma'(s)$ for all $s \in [0,1]$, and $H(0,t) = x_0$, $H(1,t) = x$ for all $t \in [0,1]$. Since $f$ is continuous, $f \circ \gamma$ and $f \circ \gamma'$ are homotopic paths with fixed endpoints (via $f \circ H$). Using the universal covering property again, we have that there exist unique lifts $\overline{f \circ \gamma}$ and $\overline{f \circ \gamma'}$ that start at $\bar{y_0}$. Thus, the homotopy $f \circ H$ can be lifted to a homotopy $\overline{f \circ H} : [0,1] \times [0,1] \to \R$ between $\overline{f \circ \gamma}$ and $\overline{f \circ \gamma'}$ with fixed endpoints. Since $\R$ is simply connected, the two lifted paths must coincide at every point and in particular at $\gamma(1)=\gamma'(1)$, implying that the definition of $\bar{f}(x)$ is independent of the choice of path $\gamma$. Finally, we show that $\bar{f}$ is continuous. Since $f$ is continuous and $p$ is the universal covering, for each $x \in \mathbb{S}^2$ there exists an open neighborhood $\mathcal{U}$ of $f(x)$ in $\bS^1$ such that any lift of $f|_{\mathcal{U}}$ to $\R$ is continuous. Therefore, $\bar f$ is continuous at every point of $\mathbb{S}^2$.
	\end{proof}

\subsection{Approximating equations}

Now we switch to the proof of Proposition \ref{conv:interm-to-heat}. We first consider one more approximation, namely we want to compare \eqref{eq:intermediate} to a linear integral-differential equation, $\wt u^{I,\ep}:[0,\infty)\times\bS^2\to\R$:
	\begin{equation}
		\begin{cases}\label{eq:linear}
			\displaystyle{\frac{d}{dt}}\wt u^{I,\ep}(t,x) &=\displaystyle{ \frac{1}{\mathsf c_2 \ep^2}}\int_{\bS^2} \left(\wt u^{I,\ep}(t,y)-\wt u^{I,\ep}(t,x)\right)K(\ep^{-1}\|x-y\|^2)d\sigma(y)\\[7pt]
			\wt u^{I,\ep}(0,x)&=u_0(x).
		\end{cases}
	\end{equation}
	
Our first task is to prove the existence and uniqueness of solutions to \eqref{eq:intermediate} and \eqref{eq:linear}. To that end, we consider a more general integral-differential equation of the form 
    \begin{equation}
		\begin{cases}\label{eq:J}
			\displaystyle{\frac{d}{dt}}\wt u^{I,\ep}(t,x) &=\displaystyle{ \frac{1}{\mathsf c_2 \ep^2}}\int_{\bS^2} J\left(\wt u^{I,\ep}(t,y)-\wt u^{I,\ep}(t,x)\right)K(\ep^{-1}\|x-y\|^2)d\sigma(y)\\[7pt]
			\wt u^{I,\ep}(0,x)&=u_0(x),
		\end{cases}
	\end{equation}
    where $J\in C^2(\R)$ is such that $J(0)=0, |J'(x)|\le1, |J''(x)|\le1$ for all $x$. We follow a fixed point procedure as in \cite{cirelli2024scaling}; let us integrate \eqref{eq:J} with respect to time to get
	\begin{equation}\label{fixedpoint}
		\wt u^{I,\ep}(t,x) = u_0(x) + \frac{1}{\mathsf c_2 \ep^2} \int_{0}^{t} \int_{\bS^2} J\left(\wt u^{I,\ep}(s,y)-\wt u^{I,\ep}(s,x)\right)K(\ep^{-1}\|x-y\|^2)d\sigma(y) ds.
	\end{equation}

	We see that finding a solution of the integral equation \eqref{eq:J} is equivalent to finding
	\[
	\wt u^{I, \ep} \in C\left([0, \infty), C^{1}(\bS^{2})\right)
	\]
	satisfying \eqref{fixedpoint}.
	
	\begin{proposition}\label{prop.existence}
		Fix any $\ep>0$. For any smooth function $u_0\in C^1(\bS^2)$, there exists a unique function $\wt u^{I, \ep} \in C\left([0, \infty), C^{1}(\bS^{2})\right)$ satisfying \eqref{fixedpoint} and hence a unique solution of \eqref{eq:J}.
	\end{proposition}
	
	\begin{proof}
		Solutions of \eqref{fixedpoint} are fixed points of the operator
		\[
		F_{u_0}(u)(t,x) = u_0(x) + \frac{1}{\mathsf c_2 \ep^2} \int_{0}^{t} \int_{\bS^2} J\left(u(s,y)- u(s,x)\right) K(\ep^{-1}\|x-y\|^2)d\sigma(y) ds,
		\]
		For a fixed initial condition $u_0 \in C^{1}$ and a positive $T$ we consider a closed ball in the Banach space
		\[
		\mathcal{X}_T := \big\{f \in C\left([0, T], C^{1}(\bS^{2})\right) : f \big|_{t=0} = u_0; \; \sup_{t \in [0,T]} \|f(t, \cdot)\|_{C^1(\bS^2)}\le 1+ \red{2}\|u_0\|_{C^1(\bS^2)}\big\},
		\]
		with the norm
		\[
		\|f\|_{\mathcal{X}_{T}} := \sup_{t \in [0,T]} \|f(t, \cdot)\|_{C^1(\bS^2)}.
		\]
		As every point $x \in \bS^{2}$ has a spherical coordinate representation 
		\[ 
		x = \left(\sin(\theta)\cos(\varphi), \sin(\theta)\sin(\varphi), \cos(\theta)\right), \quad \theta \in [0, \pi], \varphi \in [0, 2\pi),
		\] for a function $f\in C^{1}(\bS^2)$ we compute its norm as
		\[
		\|f\|_{C^1(\bS^2)} = \|f\|_{L^{\infty}(\bS^2)} + \left\| \frac{\partial f}{\partial \theta} \right\|_{L^{\infty}(\bS^2)} + \left\| \frac{\partial f}{\partial \varphi} \right\|_{L^{\infty}(\bS^2)}.
		\]
		We remark that this norm is independent of the coordinate system chosen to work on $\mathbb S^2$. Our plan is to apply Banach's fixed point theorem to $F_{u_0}$ in $\mathcal{X}_{T}$. We must check:
		\begin{enumerate}
			\item[$(a)$] $F_{u_0}(\mathcal{X}_{T}) \subseteq \mathcal{X}_{T}$;
			\item[$(b)$] $F_{u_0}$ is a contraction, i.e.~there exists $\nu \in (0, 1)$ such that
			\begin{equation*}
				\|F_{u_0} (u) - F_{u_0}(v) \|_{\mathcal{X}_{T}} \leq \nu \| u - v \|_{\mathcal{X}_{T}},
			\end{equation*}
			for all $u, v \in \mathcal{X}_{T}$.
		\end{enumerate}
	    \red{Let us first show part $(b)$, since the same computations will be useful for proving part $(a)$}. In the following arguments, we will assume the existence of a constant $M > 0$ such that $|K(r)| \leq M$ and $|K'(r)| \leq M$ for all $r > 0$. This is fine when $K$ has a continuous derivative but it is not if $K$ is the indicator function of $[0,1]$. Throughout this proof we assume $K$ is smooth and the case in which $K$ is an indicator is treated in the Appendix. To begin, applying the mean-value theorem and $|J'|\le 1$, we have
		\begin{align*}
			\left| F_{u_0}(u) - F_{u_0}(v) \right| 
            &\le \frac{1}{\mathsf c_2 \ep^2} \int_{0}^{t}  \int_{\bS^2}\left| J\left( u(s,y) - u(s,x) \right) -J\left( v(s,y) - v(s,x) \right) \right| K(\ep^{-1}\|x-y\|^2) d \sigma(y) ds\\
			&\le \frac{1}{\mathsf c_2 \ep^2}\int_{0}^{t} \int_{\bS^2}  \left| 
			\left( u(s,y) - v(s,y) \right)- \left( u(s,x) - v(s,x) \right) \right| K(\ep^{-1}\|x-y\|^2) d \sigma(y) ds  \\
			&\leq \frac{1}{\mathsf c_2 \ep^2} \int_{0}^{t} \sigma(\bS^2) \cdot 2M  \|u(s,\cdot) - v(s, \cdot)\|_{C^{1}(\bS^{2})}  \\
			&\leq \frac{T \cdot \sigma(\bS^2) \cdot 2M}{\mathsf c_2 \ep^2} \|u-v\|_{\mathcal{X}_{T}}.
		\end{align*}
		
		Regarding the derivatives, after noticing that $x$ is the only variable that depends on $\theta$, we have that
		\begin{align*}
		\frac{\partial F_{u_0}(u)}{\partial \theta} = \frac{\partial u_0}{\partial \theta} + \frac{1}{\mathsf c_2 \ep^2} \int_{0}^{t} \int_{\bS^{2}} \Big[&-J'((u(s,y) - u(s,x))\frac{\partial u}{\partial \theta}(s,x) \wt K(\theta,\varphi,y) \\
        &+ J(u(s,y) - u(s,x)) \frac{\partial \wt K(\theta,\varphi,y)}{\partial \theta}\Big] d \sigma(y) ds.
		\end{align*}
		Here $\wt K(\theta, \varphi, y)= K(\ep^{-1}\|x(\theta,\varphi)-y\|^2)$. Regarding $\wt K$, 
		\begin{align*}
			\left| \frac{\partial \wt{K}(\theta, \varphi, y)}{\partial \theta} \right| &= \left| K'\left(\ep^{-1}{\|x-y\|^{2}}\right) \cdot \ep^{-1} \cdot \frac{\partial{\|x-y\|^2}}{\partial \theta} \right| \\
			&\le \left| K'\left(\ep^{-1}{\|x-y\|^{2}}\right) \right|\cdot 2\ep^{-1} \|x-y\| \cdot \left\| \frac{\partial x}{\partial\theta} \right\| \le \frac{2M}{\sqrt{\ep}}.
		\end{align*}
        \begin{align*}
            \left| \frac{\partial \wt{K}(\theta,\varphi,y)}{\partial \varphi} \right| &= \left| K'\left(\ep^{-1}\|x-y\|^{2}\right) \cdot\ep^{-1} \cdot \frac{\partial{\|x-y\|^2}}{\partial \varphi} \right| \\
			&\le \left| K'\left(\ep^{-1}\|x-y\|^{2}\right) \right|\cdot 2\ep^{-1} \|x-y\| \cdot \left\|\frac{\partial x}{\partial\varphi} \right\|\le \frac{2M}{\sqrt{\ep}}.
        \end{align*}
		
		Then,
		\begin{align}
			\nonumber \left|\frac{\partial F(u)}{\partial \theta} - \frac{\partial F(v)}{\partial \theta} \right| 
            & = \Big| \frac{1}{\mathsf c_2 \ep^2} \int_{0}^{t} \int_{\bS^{2}} \Big[\big( -J'((u(s,y) - u(s,x))+ J'((v(s,y) - v(s,x))\big)\frac{\partial u}{\partial \theta}(s,x) \wt K(\theta, \varphi,y)\\
			\nonumber  & -J'((v(s,y) - v(s,x))\left(\frac{\partial u}{\partial \theta}(s,x) - \frac{\partial v}{\partial \theta}(s,x) \right) \wt{K}(\theta,\varphi,y) \\
			\nonumber &\quad + \big( J\left( u(s,y) - u(s,x) \right) - J\left( v(s,y) - v(s,x) \right) \big) \frac{\partial \wt{K}}{\partial \theta}(\theta,\varphi,y)\Big] d \sigma(y) ds \Big| \\
            \nonumber &\leq \frac{1}{\mathsf c_2 \ep^2} \int_{0}^{t} \sigma(\bS^2)\Big( (\red{3}+ \red{4}\|u_0\|_{C^1(\bS^2)})
            \|u(s,\cdot) - v(s, \cdot)\|_{C^{1}(\bS^{2})} \cdot M  \\
			\nonumber &\phantom{\leq \frac{1}{\mathsf c_2 \ep^2} \int_{0}^{t} \int_{\bS^{2}} \|u(s,\cdot) -} + \|u(s,\cdot) - v(s, \cdot)\|_{C^{1}(\bS^{2})} \frac{\red{4}M}{\sqrt{\ep}} \Big) ds \\
			&\leq \frac{T \cdot \sigma(\bS^2) \cdot 2 C}{\mathsf c_2 \ep^2} \|u-v\|_{\mathcal{X}_{T}}, \label{eq:bound.derivative}
		\end{align}
		where $C := \max \left\{M(\red{3}+\red{4}\|u_0\|_{C^1(\bS^2)}), \frac{\red{4}M}{\sqrt{\ep}}\right\}$. Here we used mean-value theorem, $|J'|, |J''| \le 1$ and $|\frac{\partial u}{\partial\theta}(t,x)|\le 1+\red{2}\|u_0\|_{C^1(\bS^2)}$.
		
		Similarly, we have that
		\[
		\left|\frac{\partial F(u)}{\partial \varphi} - \frac{\partial F(v)}{\partial \varphi} \right| \leq\frac{T \cdot \sigma(\bS^2) \cdot 2 C}{\mathsf c_2 \ep^2} \|u-v\|_{\mathcal{X}_{T}}.
		\]
		
		Finally, by combining these three bounds, we obtain
		\[
		\|F_{u_0} (u) - F_{u_0}(v) \|_{\mathcal{X}_{T}} \leq \frac{T \cdot \sigma(\bS^2) \cdot 6C}{\mathsf c_2 \ep^2} \| u - v \|_{\mathcal{X}_{T}}.
		\]
		
       Choosing $T_1 := \frac{\mathsf c_2 \ep^2}{\red{\sigma(\bS^2)} \cdot \red{24}C}$ we get that the map $F_{u_0}$ is a contraction on $\mathcal X_{T_1}$.

       \red{Observe that $F_{u_0}(0)(t,x) = u_0(x)$. Therefore, if we take $v \equiv 0$ and $u \in \mathcal{X}_{T_1}$ in the previous inequality we obtain:
       \begin{align*}
       \|F_{u_0}(u)\|_{\mathcal{X}_{T_1}} &\leq \frac{T_1 \cdot \sigma(\bS^2) \cdot 6C}{\mathsf c_2 \ep^2} \| u\|_{\mathcal{X}_{T_1}} + \|u_0\|_{C^{1}(\bS^2)} \\
       &\leq \frac{1}{4} \left(1 + 2\|u_0\|_{C^1(\bS^2)}\right) + \|u_0\|_{C^1(\bS^2)} \\
       &\leq 1 + 2\|u_0\|_{C^1(\bS^2)}.
       \end{align*}
       Moreover, it follows that if $u \in \mathcal{X}_{T_1}$ then $F_{u_0}(u)\big|_{t=0} = u_0$. Thus, we have proved part $(a)$.}
       
       Hence, by Banach's fixed point theorem, we obtain existence of a unique solution to \eqref{fixedpoint} in the time interval $[0, T_1]$. Since $1/T_1$ depends linearly on $\|u_0\|_{C^1(\bS^2)}$ and $u$ grows at most by one unit in $[0,T_1]$, we can iterate this procedure up to any time $T$, to obtain a solution in $[0,T]$. More precisely, we can construct a sequence of times $T_i$, and the existence of a solution in $[T_i,T_{i+1}]$ for every $i\ge 1$. The above discussion guarantees that
       \[
       T_{i+1} = \sum_{j=0}^iT_{j+1} - T_j \ge  \sum_{j=0}^i \frac{\mathsf c}{j}  \to \infty,
       \]
       and hence exceeds every $T<\infty$ if $i$ is large enough.
       \end{proof}


Proposition \ref{conv:interm-to-heat} is going to be obtained as a consequence of Proposition \ref{prop.convergence} and Lemma \ref{lem:linearize} below.

	\begin{lemma}\label{lem:linearize}
		Fix any finite $T$. Then there exists some finite constant $C=C(T, K)$ such that
		\begin{align*}
			\left\|u^{I,\ep}-\wt u^{I,\ep}\right\|_{L^\infty([0,T]\times\bS^2)}\le C\sqrt{\ep}.
		\end{align*}
	\end{lemma}


	\begin{proposition}\label{prop.convergence}
		Let $u:[0,\infty)\times\bS^2\to\R$ be the solution of \eqref{eq:heat} and $\wt u^{I,\ep}:[0,\infty)\times \bS^2\to\R$ the solution of \eqref{eq:linear} {}{with $u_0\in C^{2,\alpha}(\bS^2)$ for some $\alpha\in(0,1)$}. Then for any $T>0$ there exists $C=C(T, {}{\alpha, u_0})>0$ and $\ep_0\in(0,1)$ such that for any $\ep<\ep_0$,
		\begin{align*}
			\left\|\wt u^{I,\ep}-u\right\|_{L^\infty([0,T]\times\bS^2)}\leq C\ep^{\alpha/2}.
		\end{align*}
	\end{proposition}
    
To prove Proposition \ref{prop.convergence} and Lemma \ref{lem:linearize}, we make some preparations. For a function $f:\bS^2\to\R$, we define its Lipschitz norm 
	\begin{align}\label{def:Lip}
		\|f\|_{\text{Lip}}:=\sup_{x\neq \bar x\in\bS^2}\frac{|f(x)-f(\bar x)|}{\|x-\bar x\|}.
	\end{align}
	In order to compare the solutions of \eqref{eq:kuramoto} and \eqref{eq:intermediate}, and of \eqref{eq:intermediate} and \eqref{eq:linear}, 
	we need the Lipschitz estimate for $u^{I,\ep}(t,\cdot)$ \red{provided in Proposition \ref{ppn:lip}}, which is uniform in $\ep\in(0,1)$ and $t\in[0,T]$. To get the bound, we first establish a comparison principle. The proof can be found in \cite[Lemma 2.3]{cirelli2024scaling}.
    
	\begin{lemma}[Comparison principle]\label{lem:comparison}
		Fix $T$ finite. Let $\Psi(x,y): \bS^2\times\bS^2\to\R_{\ge0}$ satisfy $\Psi(x,y)>0$ whenever $\|x-y\|< \sqrt{\ep}$, and $v, w:[0, T]\times\bS^2\to\R$ be two continuous functions with continuous time derivative that satisfy
		\begin{align*}
			\frac{d}{dt}v - \frac{1}{\mathsf c_2 \ep^2}\int_{\bS^2}\Psi(x,y)\left(v(t,y)-v(t,x)\right) d \sigma(y) &\geq \frac{d}{dt}w-\frac{1}{\mathsf c_2 \ep^2}\int_{\bS^2}\Psi(x, y)\left(w(t,y)-w(t,x)\right) d \sigma(y)  \\[5pt]
			v|_{t=0}&\geq  w|_{t=0} .
		\end{align*}
		Then, we have
		\[
		v(t,x)\geq w(t,x), \quad \forall t\in[0,T],\; x\in\bS^2.
		\]
	\end{lemma}

	\begin{proposition}[Uniform Lipschitz bound]\label{ppn:lip}
		Let $T\in(0,\infty)$ be fixed and $u_0\in C^1(\bS^2)$. There exist $\ep_0=\ep_0(T,u_0)>0$ and $C_T = C_T (u_0)$ finite such that if $\ep < \ep_0$ and $u^{I,\ep}$ is the unique $C([0,T]; C^1(\bS^2))$ solution of \eqref{eq:intermediate}, then 
		\begin{align*}
			\sup_{t\in[0,T]}\left\|u^{I,\ep} (t, \cdot)\right\|_{\rm{Lip}} \le C_T.
		\end{align*}
	\end{proposition}

    \begin{proof}

We call {\em axis} an infinite line through the origin in $\R^3$. Perpendicular to any axis, there is a plane that passes through the origin, and intersects with $\bS^2$ at a circle. We call this circle {\em equator} (corresponding to that axis). Given any two points $x, \bar x\in\bS^2$, we can find a unique axis and equator such that there is a rotation $\phi_{x,\bar x}:\bS^2\to\bS^2$ around this axis that leaves the equator invariant and sends $x$ to $\bar x$. For any rotation $\phi:\bS^2\to\bS^2$, there is a unique equator that is invariant under $\phi$ which we call $E(\phi)$, and we  write $\mathsf c_\phi:=\|\phi(x)-x\|$ for any (and all) $x\in E(\phi)$. Furthermore, we have
\begin{align}\label{rot-angle}
\|\phi(y)-y\| \le \mathsf c_\phi, \quad \forall y\in\bS^2.
\end{align}

We will prove that for any rotation $\phi: \bS^2\to \bS^2$, and any $\wt x\in E(\phi)$,
		\begin{align}\label{bound-Lip}
			\frac{|u^{I,\ep}(t,\phi(\wt x))-u^{I,\ep}(t,\wt x)|}{\|\phi(\wt x)-\wt x\|}\le C_T,
		\end{align}
		where $C_T$ does not depend on $\phi$ or $\wt x$, and is uniform for $t\in[0,T]$. It can be readily seen that the Lipschitz norm we want to bound \eqref{def:Lip} can be expressed as
		\begin{align}\label{alter-lip}
			\|u^{I,\ep}(t,\cdot)\|_{\rm{Lip}}=\sup_{\phi,\; \wt x\in E(\phi)}\frac{|u^{I,\ep}(t,\phi(\wt x))-u^{I,\ep}(t,\wt x)|}{\|\phi(\wt x)-\wt x\|},
		\end{align}
		where the supremum is over all rotations $\phi:\bS^2\to\bS^2$ and all $\wt x\in E(\phi)$, hence \eqref{bound-Lip} is all we need to prove.
		
		To this end, since $\phi$ is a $1$-to-$1$ map, each $y\in \bS^2$ can be written as $y = \phi (y')$ for some $y'\in \bS^2$. For any $x\in\bS^2, t>0$, we have
		\begin{align*}
			\displaystyle{\frac{d}{dt}}u^{I,\ep}(t, \phi(x)) &=\displaystyle{ \frac{1}{\mathsf c_2 \ep^2}}\int_{\bS^2}\sin\left(u^{I,\ep}(t,y)-u^{I,\ep}(t, \phi(x))\right)K\Big(\frac{\|\phi(x)-y\|^2}{\ep}\Big)d\sigma(y)\\
			&=\displaystyle{ \frac{1}{\mathsf c_2 \ep^2}}\int_{\bS^2}\sin\left(u^{I,\ep}(t,\phi (y'))-u^{I,\ep}(t, \phi(x))\right)K(\ep^{-}\|\phi(x)-\phi(y')\|^2)d\sigma(\phi(y')).
		\end{align*}
		By the property of rotation, we have $\|\phi(x)-\phi(y')\|= \|x-y'\|$ and $d\sigma(\phi(y'))=d\sigma(y')$ i.e.~the Jacobian determinant is $1$. Hence, we have 
		\begin{align*}
			\displaystyle{\frac{d}{dt}}u^{I,\ep}(t, \phi(x)) =\displaystyle{ \frac{1}{\mathsf c_2 \ep^2}}\int_{\bS^2}\sin\left(u^{I,\ep}(t,\phi (y'))-u^{I,\ep}(t, \phi(x))\right)K(\ep^{-1}\|x-y'\|^2)d \sigma(y').
		\end{align*}
		Since $u^{I, \ep}(t,x)$, $x\in\bS^2, t>0$ satisfies 
		\begin{align*}
			\displaystyle{\frac{d}{dt}}u^{I,\ep}(t, x) =\displaystyle{ \frac{1}{\mathsf c_2 \ep^2}}\int_{\bS^2}\sin\left(u^{I,\ep}(t,y)-u^{I,\ep}(t, x)\right)K(\ep^{-1}\|x-y\|^2)d \sigma(y),
		\end{align*}
		taking the difference of preceding two equations, and upon calling $w(t, x):=u^{I,\ep}(t, \phi(x))-u^{I, \ep}(t,x)$ we have 
		\begin{align*}
			&\displaystyle{\frac{d}{dt}} w(t,x)=\displaystyle{\frac{d}{dt}}\big(u^{I,\ep}(t, \phi(x))-u^{I, \ep}(t,x)\big) \\
			&=\displaystyle{ \frac{1}{\mathsf c_2 \ep^2}}\int_{\bS^2}\Big[\sin\left(u^{I,\ep}(t,\phi (y))-u^{I,\ep}(t, \phi(x))\right)-\sin\left(u^{I,\ep}(t,y)-u^{I,\ep}(t, x)\right)\Big]K(\ep^{-1}\|x-y\|^2)d \sigma(y),
		\end{align*}
		with initial condition $w(0,x)=u_0(\phi(x))-u_0(x)$. We can write 
		\begin{align*}
			\sin\left(u^{I,\ep}(t,\phi (y))-u^{I,\ep}(t, \phi(x))\right)-\sin\left(u^{I,\ep}(t,y)-u^{I,\ep}(t, x)\right)&=\Psi(x,y)\big(w(t,y)-w(t,x)\big),
		\end{align*}
		for 
		\begin{align*}
			\Psi(x,y):=\int_0^1 \cos\Big(s\big(u^{I,\ep}(t,\phi (y))-u^{I,\ep}(t, \phi(x))\big)+(1-s)\big(u^{I,\ep}(t,y)-u^{I,\ep}(t, x)\big)\Big) \, ds,
		\end{align*}
		whereby
		\begin{align}\label{eq:w}
			&\displaystyle{\frac{d}{dt}} w(t,x)=\displaystyle{ \frac{1}{\mathsf c_2 \ep^2}}\int_{\bS^2}\Psi(x,y)\big(w(t,y)-w(t,x)\big)K(\ep^{-1}\|x-y\|^2)d \sigma(y).
		\end{align}
		In order to apply the comparison principle, we need to ensure $\Psi(x,y)>0$ whenever $\|x-y\| \le \sqrt{\ep}$. For some large constant $M$ (to be specified), let us define a time
		\begin{align*}
			\tau_M:=\inf\big\{t\ge0:\, \|u^{I,\ep}(t,\cdot)\|_{\rm{Lip}}\ge M\big\}.
		\end{align*}
		In the time interval $[0,\tau_M)$, the Lipschitz norm of $u^{I,\ep}(t,\cdot)$ is controlled, and hence for any $x,y\in\bS^2$ such that $\|x-y\|\le\sqrt{\ep}$
		we have
		\begin{align*}
			|u^{I,\ep}(t,\phi (y))-u^{I,\ep}(t, \phi(x))|&\le M\|\phi(y)-\phi(x)\| = M\|y-x\|\le M\sqrt{\ep},\\
			|u^{I,\ep}(t,y)-u^{I,\ep}(t, x)|&\le M\|y-x\|\le M\sqrt{\ep}.
		\end{align*}
		Hence upon taking $\ep$ small enough (e.g.~$M\sqrt{\ep} < \frac{\pi}{4}$), we can ensure the argument of the cosine in the definition of $\Psi(x,y)$ is less than $\frac{\pi}{4}$ and thus $\Psi(x,y)>0$.
		
		Let us use the {\em barrier} function 
		\begin{align*}
			\ovl w(t,x):=(t+1)\|u_0\|_{\rm{Lip}}\mathsf c_\phi, \quad t\ge0, \, x\in\bS^2.
		\end{align*}
        Since $u_0$ is Lipschitz, for any $x\in\bS^2$ we have 
		\begin{align*}
			|w(0,x)|=|u_0(\phi(x))-u_0(x)|\le \|u_0\|_{\rm{Lip}}\|\phi(x)-x\|\le \|u_0\|_{\rm{Lip}}\mathsf c_\phi=\ovl w(0,x),
		\end{align*}
        where we used \eqref{rot-angle} in the second inequality.  We also have for any $x\in\bS^2, t>0$,
		\begin{align*}
			&\displaystyle{\frac{d}{dt}} \ovl w(t,x) -\displaystyle{ \frac{1}{\mathsf c_2 \ep^2}}\int_{\bS^2}\Psi(x,y)\big(\ovl w(t,y)-\ovl w(t,x)\big)K(\ep^{-1}\|x-y\|^2)d \sigma(y) = \displaystyle{\frac{d}{dt}} \ovl w(t,x) =\|u_0\|_{\rm{Lip}}\mathsf c_\phi\\
			&>0 = \displaystyle{\frac{d}{dt}} w(t,x)-\displaystyle{ \frac{1}{\mathsf c_2 \ep^2}}\int_{\bS^2}\Psi(x,y)\big(w(t,y)-w(t,x)\big)K(\ep^{-1}\|x-y\|^2)d \sigma(y).
		\end{align*}
		since $\ovl w$ is space-independent and we used \eqref{eq:w}. Now by Lemma \ref{lem:comparison}, we have
		\begin{align*}
			w(t,x)\le (t+1)\|u_0\|_{\rm{Lip}}\mathsf c_\phi, \quad \forall t\in[0,\tau_M).
		\end{align*}
		The same argument applied to $-w(t,x)$ yields
		\begin{align*}
			-w(t,x)\le (t+1)\|u_0\|_{\rm{Lip}}\mathsf c_\phi, \quad \forall t\in[0,\tau_M).
		\end{align*}
		In other words, for $t\in[0,\tau_M)$, we have
		\begin{align*}
			\sup_{x\in\bS^2}\frac{|u^{I,\ep}(t, \phi(x))-u^{I, \ep}(t,x)|}{\mathsf c_\phi}\le (t+1)\|u_0\|_{\rm{Lip}}.
		\end{align*}
		Since for any $\wt x\in E(\phi)$, we have $\|\phi(\wt x)-\wt x\|=\mathsf c_\phi$, 
		\begin{align*}
			\frac{|u^{I,\ep}(t, \phi(\wt x))-u^{I, \ep}(t,\wt x)|}{\|\phi(\wt x)-\wt x\|}\le (t+1)\|u_0\|_{\rm{Lip}},
		\end{align*}
		where the right-hand side is uniform for all rotations $\phi$ and $\wt x\in E(\phi)$. Namely, we have
		\begin{align*}
			\sup_{\phi, \; \wt x\in E(\phi)}\frac{|u^{I,\ep}(t, \phi(\wt x))-u^{I, \ep}(t,\wt x)|}{\|\phi(\wt x)-\wt x\|}\le (t+1)\|u_0\|_{\rm{Lip}},
		\end{align*}
		in the time interval $t\in[0,\tau_M)$. In view of \eqref{alter-lip} it remains to take $M>2(T+1)\|u_0\|_{\rm{Lip}}$ so that $\tau_M>T$, and obtain 
		\begin{align*}
			\sup_{t\in[0,T]}\|u^{I,\ep} (t, \cdot)\|_{\rm{Lip}}\le C_T=(T+1)\|u_0\|_{\rm{Lip}}.
		\end{align*}
	\end{proof}

	\begin{proof}[Proof of Lemma \ref{lem:linearize}]
		By Proposition \ref{ppn:lip}, we have $|u^{I,\ep}(t,y)-u^{I,\ep}(t,x)|\le C_T\|x-y\|\le C_T\sqrt{\ep}$ for any $t\in[0,T]$ and $x,y\in\bS^2$ such that $\|x-y\|^2\le\ep$.  Since $\sin x = x+ O(x^3)$ for $x$ close to $0$, we have that 
		\begin{align*}
			\displaystyle{\frac{d}{dt}}u^{I,\ep}(t,x) &=\displaystyle{ \frac{1}{\mathsf c_2 \ep^2}}\int_{\bS^2}\sin\left(u^{I,\ep}(t,y)-u^{I,\ep}(t,x)\right)K(\ep^{-1}\|x-y\|^2)d \sigma(y)\\
			&\le\displaystyle{ \frac{1}{\mathsf c_2 \ep^2}}\int_{\bS^2}\left(u^{I,\ep}(t,y)-u^{I,\ep}(t,x)\right)K(\ep^{-1}\|x-y\|^2)d \sigma(y)\\
			&\quad +\displaystyle{ \frac{1}{\mathsf c_2 \ep^2}}\int_{\bS^2} C(C_T\sqrt{\ep})^3 K(\ep^{-1}\|x-y\|^2)d \sigma(y)\\
			&=\displaystyle{ \frac{1}{\mathsf c_2 \ep^2}}\int_{\bS^2}\left(u^{I,\ep}(t,y)-u^{I,\ep}(t,x)\right)K(\ep^{-1}\|x-y\|^2)d \sigma(y) +C\sqrt{\ep}.
		\end{align*}
		Denote $w= u^{I,\ep}-\wt u^{I,\ep}$. From the above expression and \eqref{eq:linear} we  have
		\[
		\displaystyle{\frac{d}{dt}}w(t,x) -\displaystyle{ \frac{1}{\mathsf c_2 \ep^2}}\int_{\bS^2}\left(w(t,y)-w(t,x)\right)K(\ep^{-1}\|x-y\|^2)d \sigma(y) \le C\sqrt{\ep}
		\]
		with $w|_{t=0}=0$. On the other hand, considering the barrier $\ovl w(t,x)=C\red{\sqrt{\ep}} t$ which is $x$-independent, we have by the comparison principle Lemma \ref{lem:comparison},
		\[
		w(t,x)= \left(u^{I,\ep}-\wt u^{I,\ep}\right)(t,x) \le C\sqrt{\ep} t.
		\]
		Analogous argument applied to $-w$ yields that 
		\[
		-w(t,x)= \left(\wt u^{I,\ep}-u^{I,\ep}\right)(t,x) \le C\sqrt{\ep} t.
		\]
		Taken together, the claim is proved.
	\end{proof}
	
	To prove Proposition \ref{prop.convergence}, we need the pointwise convergence of the linear integral operator to the Laplace-Beltrami operator. Results of this type are well known in the literature, e.g.~\cite[Eq. (1.4)]{singer2006graph}, \cite[Lemma 8]{coifman2006diffusion}, since they appear in machine learning contexts.
	
	\begin{lemma}[Pointwise convergence of operator]\label{lem:conv-op}
		Let $v\in C^{2,\alpha}{(\bS^2)}$ for some  $\alpha\in(0,1)$, then there exists some finite constant $\wt C$ depending only on the $C^{2,\alpha}$-norm of $v$, such that
		\begin{align*}
		\sup_{x\in\bS^2}\Big|\frac{1}{\mathsf c_2 \ep^2}\int_{\bS^2}\left(v(y)-v(x)\right)K(\ep^{-1}\|x-y\|^2) d \sigma(y) -\kappa\, \Delta_{\bS^2} v(x) \Big|\le \wt C\ep^{\alpha/2}.
		\end{align*}
	The constant $\kappa$ is defined in \eqref{cst:kappa}.
	\end{lemma}
	
	We are ready to prove Proposition \ref{prop.convergence}.
    
	\begin{proof}[Proof of Proposition \ref{prop.convergence}]
		Since $u_0\in C^{2, \alpha}(\bS^2)$ for some $\alpha\in(0,1)$, parabolic regularity theory guarantees that the unique solution of the heat equation $u$ is $C^{2,\alpha}$ in space and $C^{1,\alpha/2}$ in time, cf. \cite{krylov} for the Euclidean setting but it extends to the sphere case. For the manifold case, see also \cite{milgram1951}.
		Let $w^\ep:=\wt u^{I,\ep}-u$. Then the function $w^\ep$ satisfies
		\begin{align}\label{eq:error}
		\frac{d}{dt}w^\ep = \frac{1}{\mathsf c_2 \ep^2}\int_{\bS^2}\left(w^\ep(t,y)-w^\ep(t,x)\right)K(\ep^{-1}\|x-y\|^2)d \sigma(y) + {\mathrm{Err}}_\ep(u)
		\end{align}
		with $ w^\ep|_{t=0}=0$, where
		\begin{align*}
			{\mathrm{Err}}_\ep(u):=\frac{1}{\mathsf c_2 \ep^2}\int_{\bS^2}\left(u(t,y)-u(t,x)\right)K(\ep^{-1}\|x-y\|^2) d \sigma(y) -\kappa\Delta_{\bS^2} u(x).
		\end{align*}
		By Lemma \ref{lem:conv-op} and the uniform $C^{2,\alpha}$ bound for the unique solution $u(t, \cdot)$ of \eqref{eq:heat} for $t\in[0,T]$, we have $|{\mathrm{Err}}_\ep(u)|\le \wt C\ep^{\alpha/2}$ for some finite constant $\wt C=\wt C(T, u_0)$ independent of $\ep$.
		
		To use the comparison principle, let us consider the space-independent barrier function,
		\[
		\bar{w}:=C_1t\ep^{\alpha/2}+C_2\ep,
		\]
		so that,
		\[
		\frac{d}{dt}\bar w - \frac{1}{\mathsf c_2 \ep^2}\int_{\bS^2}\left(\bar w(t,y)-\bar w(t,x)\right) K(\ep^{-1}\|x-y\|^2)d \sigma(y) =\frac{d}{dt}\bar w=C_1\ep^{\alpha/2}.
		\]
		In view of \eqref{eq:error}, we can choose $C_1>\wt C$ and $C_2>0$ such that $\bar w$ and $w^\ep$ satisfy the assumptions of Lemma \ref{lem:comparison}. Then, Lemma \ref{lem:comparison} yields
		\[
		\wt u^{I,\ep}-u=w^\ep\leq C_1t\ep^{\alpha/2}+C_2\ep
		\]
		for all $t\in[0,T]$. An analogous reasoning with $-w^\ep$ gives the lower bound and hence proves the result.
	\end{proof}

	\section{Synchronization}
	\label{sec:sync}

    Having proved the scaling limit for \eqref{eq:kuramoto}, we are ready to prove the synchronization result.
    
	\begin{proof}[Proof of Theorem \ref{thm:sync}]
		First, note that the solution $u$ of the heat equation \eqref{eq:heat} with initial condition $u_0$ converges in $L^\infty$ norm to the constant $\iota= \iota(u_0) :=\sigma(\mathbb S^2)^{-1}\int_{\bS^2} u_0(x)d\sigma(x)$, i.e.
		\[
		\lim_{t\to \infty} \|u(t,\cdot) - \iota\|_{L^\infty(\mathbb S^2)} = 0,
		\]
		see \cite{milgram1951}. In particular, there is $T>0$ such that $\|u(T,\cdot) - \iota\|_{L^\infty(\mathbb S^2)} < \pi/8$. Let $u^n$ be the solution of \eqref{eq:kuramoto} with initial condition $u_0^n$. From our Convergence Theorem \ref{thm.main} we get the existence of a set $\tilde \Omega$ with $\P(\tilde \Omega)=1$ such that in that event it holds $\max_{1\le i \le n}|u^n(T,x_i) - \iota| < \pi/4$ for $n$ large enough. In particular, for those $n$ we have
		\begin{equation}
			\label{phase.gamma}
			\max_{1\le i,j \le n}|u^n(T,x_i) - u^n(T,x_j)| < \pi/2.
		\end{equation}
		Let $B_n = \{\mathbb G_n \text{\rm{ is connected}} \}$. In the regime given by \eqref{ep-regime} we have $\P(B_n^c) \le \exp(-\sqrt n)$ for $n$ large enough \cite{Walters}. It is well known that when $\mathbb G_n$ is connected and \eqref{phase.gamma} holds, phase synchronization is achieved \cite{dorfler2011critical, ling2020critical, bullo2020lectures}. We obtain that $\tilde \Omega \cap B_n \subset A_n$. By means of Borel-Cantelli lemma $\P((\tilde \Omega \cap B_n)^c \textrm{ i.o.})=0$ and the same holds for $A_n^c$. This finishes the proof.
	\end{proof}

Notice that Theorem \ref{thm:sync} is not enough to guarantee global synchronization since we require $u_0^n \to u_0 \in C^{2,\alpha}$. We are saying nothing about initial conditions $u_0^n$ that converge to a non-smooth function, like in Figure \ref{fig:solar-time}, or when they do not converge, like if they are given by i.i.d. random variables, or many other choices.

    \begin{figure}
    \begin{center}
    \includegraphics[width=.75\textwidth]{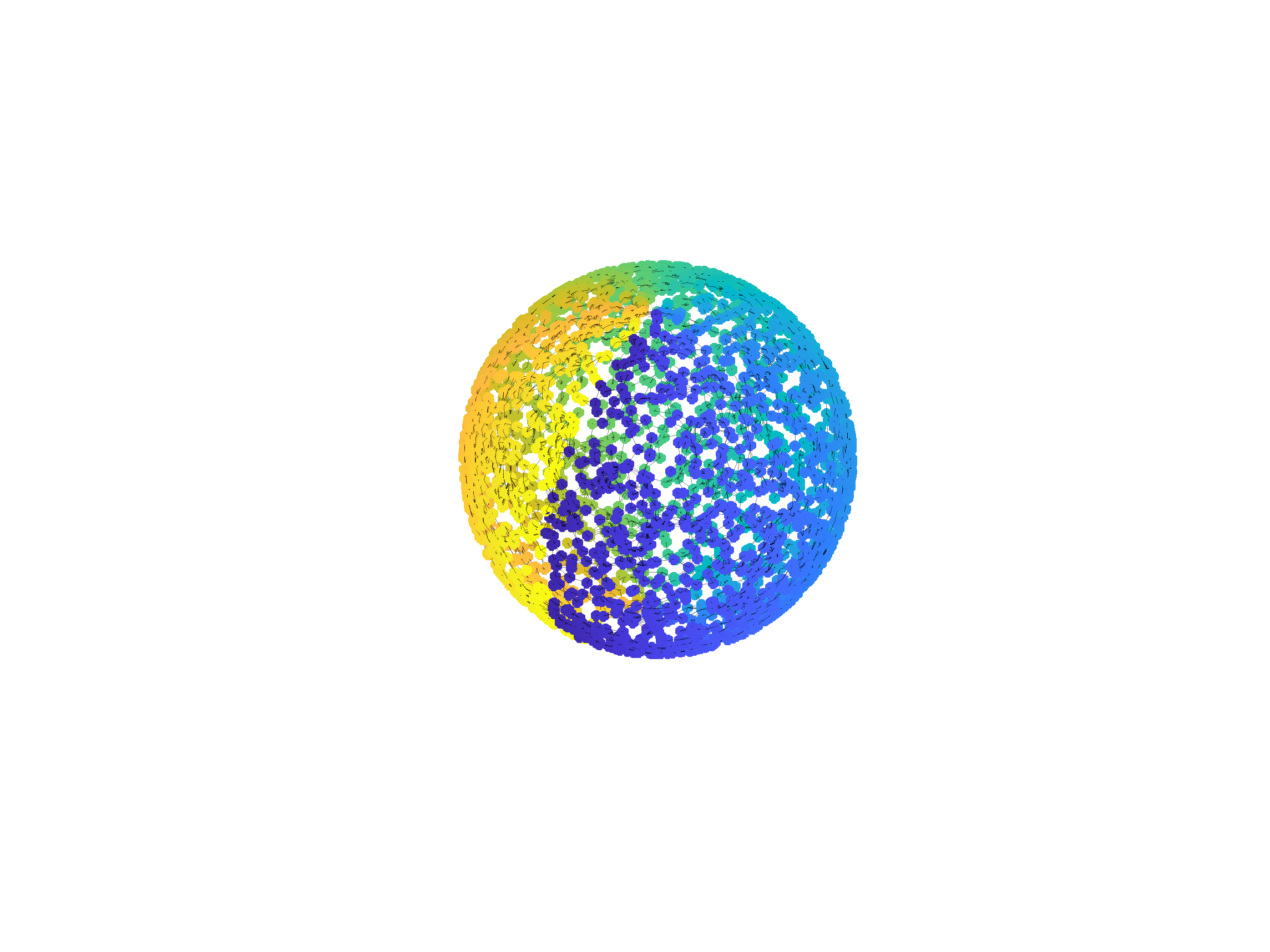}
    \end{center}
    \caption{Solar time. Different colors represent different values of $u$. Since the function is interpreted as taking values in $\mathbb S^1$, the function is continuous at the ``International Day Line'', but not at the poles. Our theorem does not discard the possibility of a stable state close to it, although numerical computations suggest it is not the case.}
    \label{fig:solar-time}
    \end{figure}

\section*{Appendix}

In this appendix we discuss a version of Proposition \ref{prop.existence} when $K$ is the indicator function of an interval, $K(r) := 1_{[0,1]}(r)$. Then $K\left(\frac{\|x-y\|^2}{\ep}\right) = 1_{\B_{\sqrt{\ep}}(x)} (y)$ and we assign equal weights to all the edges. This is an important case but is not included in the proof of Proposition \ref{prop.existence} since $K$ is not continuous (let alone $C^1$, as it is required in the proof). However, this case can be treated separately. We remark that the only place in the whole manuscript where the smoothness of $K$ is used is to obtain the bounds \eqref{eq:bound.derivative} on the derivatives of $F_{u_0}(u) -  F_{u_0}(v)$ that are used to apply Banach's fixed point theorem.

In this appendix we obtain the same kind of bounds but for $K$ being the indicator of $[0,1]$ and hence, we obtain all our theorems for that case as well.

Given our choice for $K$, the functional under consideration takes the form
		\[
		F_{u_0}(u)(t,x) = u_0(x) + \frac{1}{\mathsf c_2 \ep^2} \int_{0}^{t} \int_{\B_{\sqrt{\ep}}(x)\, \cap\, \bS^2} J\left(u(s,y)- u(s,x)\right) d \sigma(y) ds.
		\]

The goal is to bound the $C^1$ norm of $F_{u_0}(u) -  F_{u_0}(v)$ for $u, v \in \mathcal X_T$.
For smooth $f$, the norm on the derivatives coincides with the Lipschitz norm, so we use formula \eqref{alter-lip} with $f= F_{u_0}(u) -  F_{u_0}(v)$. Without loss of generality we can restrict the supremum in $\phi$ in $\eqref{alter-lip}$ to those $\phi$ with sufficiently small rotating angle such that $\B_{\sqrt{\ep}}(\phi(\tilde x))\cap\bS^2$ and $\B_{\sqrt{\ep}}(\tilde x)\cap\bS^2$ have non-empty intersection for any $\tilde x\in E(\phi)$. We denote $w=u-v$ and for such a rotation $\phi$ and any $\tilde x \in E(\phi)$ we have,

\begin{align*}\label{Lip.Indic}
\mathsf c_2 &\ep^2(f(t, \phi(\tilde x)) - f(t, \tilde x)) \\
& =\int_{0}^{t} \int_{\B_{\sqrt{\ep}}(\phi(\tilde x))\, \cap\, \bS^2} J\left(w(s,y)- w(s,\phi(\tilde x))\right) d \sigma(y) ds -   \int_{0}^{t} \int_{\B_{\sqrt{\ep}}(\tilde x)\, \cap\, \bS^2} J\left(w(s,y)- w(s,\tilde x)\right) d \sigma(y) ds\\
& =\Big(\int_{0}^{t} \int_{\B_{\sqrt{\ep}}(\phi(\tilde x))\, \cap\, \bS^2} J\left(w(s,y)- w(s,\tilde x)\right) d \sigma(y) ds -   \int_{0}^{t} \int_{\B_{\sqrt{\ep}}(\tilde x)\, \cap\, \bS^2} J\left(w(s,y)- w(s,\tilde x)\right) d \sigma(y) ds\Big)\\
& \hspace{1cm} + \int_{0}^{t} \int_{\B_{\sqrt{\ep}}(\phi(\tilde x))\,\cap\, \bS^2} \Big(J\left(w(s,y)- w(s,\phi(\tilde x))\right) -J\left(w(s,y)- w(s,\tilde x)\right) \Big)d \sigma(y) ds.
\end{align*}
Then, using $J(0)=0, |J'|\le 1$ and the mean-value theorem, we have for some finite constant $\mathsf c$ and any $t\in[0,T]$,
\begin{align*}
&\mathsf c_2 \ep^2|f(t, \phi(\tilde x)) - f(t, \tilde x)|\\
&\hspace{1cm}\le \int_0^t\sup_{y \in \mathbb S^2} \left|w(s,y)- w(s,\tilde x)\right| 2\sigma\Big(\big(\B_{\sqrt{\ep}}(\phi(\tilde x))\cap\bS^2\big)\setminus \big(\B_{\sqrt{\ep}}(\tilde x)\cap  \bS^2 \big) \Big)ds \\
&\hspace{3cm}+ \int_0^t|w(s,\phi(\tilde x))-w(s,\tilde x)| \sigma\big(\B_{\sqrt{\ep}}(\phi(\tilde x))\cap \bS^2\big)ds \\
&\hspace{1cm}\le \int_0^t2\|w\|_{L^\infty} \mathsf c \sqrt{\ep} \|\phi(\tilde x)-\tilde x\| ds+\int_0^t\|w(s, \cdot)\|_{C^1}\|\phi(\tilde x)-\tilde x\|\mathsf c T \ep \, ds \\
& \hspace{1cm}\le  \mathsf c  \sqrt{\ep}  \int_0^t\|w(s,\cdot)\|_{C^1}ds \, \|\phi(\tilde x)-\tilde x\|\\
&\hspace{1cm}\le  \mathsf c T \sqrt{\ep} \|w\|_{\cX_T} \|\phi(\tilde x)-\tilde x\|.
\end{align*}
We have obtainede
\[
\frac{|f(t,\phi(\tilde x)) - f(t,\tilde x)|}{\|\phi(\tilde x)-\tilde x\|} \le \frac{\mathsf c T}{\mathsf c_2 \ep^{3/2}}  \|u- v\|_{\mathcal X_T}.
\]
Hence, for $T$ sufficiently small
\[
\|F_{u_0}(u) -  F_{u_0}(v)\|_{\mathcal{X}_{T}} \le \nu \|u -  v\|_{\mathcal{X}_{T}}, \qquad (0< \nu<1),
\]
for all $u,v \in \mathcal{X}_{T}$. This proves that $F_{u_0}$ is a contraction in $\mathcal X_T$ and the result follows.

\centerline{{\bf Acknowledgments}}
We thank Shuyang Ling for enlightening discussions and pointers to the recent literature. Pablo Groisman and Cecilia De Vita are partially supported by CONICET Grant PIP 2021 11220200102825CO, UBACyT Grant 20020190100293BA and PICT 2021-00113 from Agencia I+D. RH is funded by the Deutsche Forschungsgemeinschaft (DFG, German Research Foundation) under Germany’s Excellence Strategy EXC 2044-390685587, Mathematics M\"unster: Dynamics-Geometry-Structure.
    
\bibliographystyle{abbrv}
\bibliography{kuramotoSL}

\end{document}